\documentclass[12pt]{article}
\usepackage{amsmath,amsthm,amsfonts,latexsym,amsopn,verbatim,amscd,amssymb}
\usepackage{pgfplots}
\usepackage{caption}
\theoremstyle{plain}
\newtheorem{theorem}{Theorem}[section]
\newtheorem{lemma}[theorem]{Lemma}

\newcommand{\bnum}{\begin{enumerate}}
\newcommand{\enum}{\end{enumerate}}

\numberwithin{equation}{section}

\DeclareMathOperator{\diam}{diam}

\begin{document}
\title{\textbf{$g$-noncommuting graph of a  finite group relative to its subgroups}}
\author{Monalisha Sharma and Rajat Kanti Nath\footnote{Corresponding author}
}
\date{}
\maketitle
\begin{center}\small{\it
Department of Mathematical Sciences, Tezpur University,\\ Napaam-784028, Sonitpur, Assam, India.\\

Emails:\, monalishasharma2013@gmail.com and rajatkantinath@yahoo.com}
\end{center}
\begin{abstract}
Let $H$ be  a  subgroup of a finite non-abelian group  $G$ and $g \in G$. Let $Z(H, G) = \{x \in H : xy = yx, \forall y \in G\}$. We introduce the graph $\Delta_{H, G}^g$   whose vertex set is $G \setminus Z(H, G)$ and two distinct vertices $x$ and $y$ are adjacent if $x \in H$ or $y \in H$ and  $[x,y] \neq g, g^{-1}$, where $[x,y] = x^{-1}y^{-1}xy$.  In this paper, we determine whether $\Delta_{H, G}^g$ is a tree among other results.
We also  discuss about its diameter and connectivity with special attention to the dihedral groups.     
\end{abstract}


\section{Introduction}
In general group theory and graph theory are closely related. Several properties of groups  can be described through properties of graphs and vice versa. Characterizations of finite groups through various graphs defined on it have been an interesting topic of research over the last five decades. Non-commuting graph is one of such interesting graphs widely studied in the literature \cite{AAM06,AF15,AI12, Daraf09, DBBM10,dn-JLTA2018,ddn2018,JDSO2015, JDS2015,JMS2019,JSO2015,Mogh05,MSZZ05,  T08,VK18}, since its inception \cite{EN76}. In this paper we introduce a generalization of non-commuting graphs of finite group. Let $H$ be  a  subgroup of a finite non-abelian group  $G$ and $g \in G$. Let $Z(H, G) = \{x \in H : xy = yx, \forall y \in G\}$. We introduce the graph $\Delta_{H, G}^g$   whose vertex set is $G \setminus Z(H, G)$ and two distinct vertices $x$ and $y$ are adjacent if $x \in H$ or $y \in H$ and  $[x,y] \neq g, g^{-1}$, where $[x,y] = x^{-1}y^{-1}xy$. Clearly, $\Delta_{H, G}^g = \Delta_{H, G}^{g^{-1}}$. Also, $\Delta_{H, G}^g$ is an induced subgraph of $\Gamma_{H, G}^g$, studied by the authors in \cite{SN2020}, induced by $G \setminus Z(H,G)$.
  If $H = G$ and $g = 1$ then $\Delta_{H, G}^g := \Gamma_G$,  the non-commuting graph of $G$.  
If $H = G$ then $\Delta_{H, G}^g := \Delta_{G}^g$, a generalization of  $\Gamma_G$ called induced g-noncommuting graph of $G$ on $G \setminus Z(G)$ studied extensively in \cite{NEGJ16,NEGJ17,NEM18} by  Erfanian and his collaborators. 

If $g \notin K(H,G)$ then any pair of vertices $(x, y)$ are adjacent in $\Delta_{H, G}^g$ trivially if $x, y \in H$ or one of $x$ and $y$ belongs to $H$. Therefore, we consider $g \in K(H,G)$.  Also, if $H=Z(H,G)$ then $K(H,G) = \{1\}$ and so $g = 1$. Thus, throughout this paper, we shall consider $H \ne Z(H,G)$ and $g \in K(H,G)$. In this paper, we determine whether $\Delta_{H, G}^g$ is a tree among other results.
We also  discuss about its diameter and connectivity with special attention to the dihedral groups. We conclude this section by the following examples of $\Delta_{H, G}^g$ where $G = A_4 = \langle a, b : a^2 = b^3 = (ab)^3 =1 \rangle$ and the subgroup $H$ is given by   $H_1 = \{1, a\}$, $H_2=\{1, bab^2\}$ or $H_3 = \{1, b^2ab\}$.
 

\begin{center}
		{\includegraphics[width=12cm]{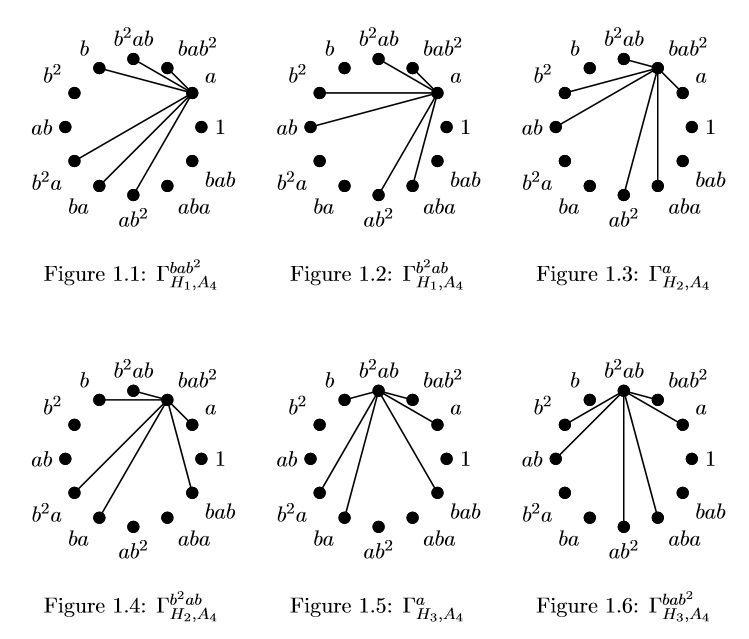}}
\end{center}

  \section{Vertex degree and a consequence}

In this section we first determine $\deg(x)$, the degree of a vertex $x$ of the graph $\Delta_{H, G}^g$. After that we determine whether  $\Delta_{H, G}^g$ is a tree.
Corresponding to Theorem 2.1 and Theorem 2.2 of \cite{SN2020} we have the following two results for $\Delta_{H, G}^g$.

\begin{theorem}\label{deg_prop_1.1}
Let $x \in H \setminus Z(H,G)$ be any vertex in  $\Delta_{H, G}^g$.  
\begin{enumerate}
\item If $g=1$ then $\deg(x)= |G| - |C_G(x)|.$	
	
\item If $g \neq 1$ and $g^2 \neq 1$  then

$\deg(x) = \begin{cases}
|G| - |Z(H,G)| - |C_G(x)| - 1, & \mbox{if $x$ is conjugate to}\\
& \mbox{$xg$ or $xg^{-1}$} \\
|G| - |Z(H,G)| - 2|C_G(x)| - 1,& \mbox{if $x$ is conjugate to}\\ 
& \mbox{$xg$ and $xg^{-1}$.}  
\end{cases}$
	
\item If $g \neq 1$ and $g^2 = 1$  then $\deg(x) = |G| - |Z(H,G)| - |C_G(x)| - 1$, whenever $x$ is conjugate to $xg$.
\end{enumerate}
\end{theorem}

\begin{proof}
(a) Let $g = 1$. Then $\deg(x)$ is the number of $y \in G \setminus Z(H,G)$ such that $xy \ne yx$. Hence, 
\[
\deg(x) = |G|-|Z(H,G)|-(|C_G(x)|-|Z(H,G)|) = |G| - |C_G(x)|.
\]
\noindent Proceeding as the proof of   \cite[Theorem 2.1 (b), (c)]{SN2020}, parts (b) and (c) follow noting that the vertex set of $\Delta_{H,G}^g$ is $G \setminus Z(H,G)$.
\end{proof}

\begin{theorem}\label{deg_prop_2.1}
Let $x \in G \setminus H$ be any vertex in  $\Delta_{H, G}^g$.  
\begin{enumerate}
\item If $g=1$ then $\deg(x)= |H| - |C_H(x)|.$
	
\item If $g \neq 1$ and $g^2 \neq 1$  then

$\deg(x) = \begin{cases}
|H| - |Z(H,G)| - |C_H(x)|, &\!\!\!\!\mbox{if $x$ is conjugate to $xg$ or}\\
&\!\!\!\mbox{$xg^{-1}$ for some element in $H$}\\
|H| - |Z(H,G)| - 2|C_H(x)|, &\!\!\!\!\mbox{if $x$ is conjugate to $xg$ and}\\
&\!\!\!\mbox{$xg^{-1}$ for some element in $H$}.
\end{cases}$

\item If $g \neq 1$ and $g^2 = 1$  then $\deg(x) = |H| - |Z(H,G)|- |C_H(x)|$, whenever $x$ is conjugate to $xg$, for some element in H.
\end{enumerate}
\end{theorem}
\begin{proof}
(a) Let $g = 1$. Then $\deg(x)$ is the number of $y \in H \setminus Z(H,G)$ such that $xy \ne yx$. Hence, 
\[
\deg(x) = |H|-|Z(H,G)|-(|C_H(x)|-|Z(H,G)|) = |H| - |C_H(x)|.
\]
\noindent Proceeding as the proof of   \cite[Theorem 2.2 (b), (c)]{SN2020}, parts (b) and (c) follow noting that the vertex set of $\Delta_{H,G}^g$ is $G \setminus Z(H,G)$.
\end{proof}

As a consequence of above results we have the following.
\begin{theorem}\label{delta-not-tree}
If $|H| \ne 2, 3, 4, 6$ then  $\Delta_{H, G}^g$ is not a tree. 
\end{theorem}
\begin{proof}
Suppose that $\Delta_{H, G}^g$ is a tree. Then there exists a vertex $x \in G \setminus Z(H,G)$ such that $\deg(x) = 1$. If $x \in H \setminus Z(H,G)$ then we have the following cases.

\noindent \textbf{Case 1:} If $g = 1$, then by Theorem \ref{deg_prop_1.1}(a), we have  $\deg(x) = |G|-|C_G(x)| = 1$. Therefore, $|C_G(x)| = 1$, contradiction.

\noindent \textbf{Case 2:} If $g \neq 1$ and $g^2 = 1$, then by Theorem \ref{deg_prop_1.1}(c), we have $\deg(x) = |G|-|Z(H,G)|-|C_G(x)|-1 = 1$. That is,  
\begin{equation}\label{induced-eq-1}
|G|-|Z(H,G)|-|C_G(x)| = 2.
\end{equation}
Therefore, $|Z(H,G)| = 1$ or $2$. Thus \eqref{induced-eq-1} gives  $|G| - |C_G(x)| = 3$ or $4$. Therefore, $|G| = 6$ or $8$.
Since $|H| \ne 2, 3, 4, 6$ we must have 
$G \cong D_8$ or $Q_8$ and $H = G$ and hence, by {\rm\cite[Theorem 2.5]{TEJ14}}, we get a contradiction.

\noindent \textbf{Case 3:} If $g \neq 1$ and $g^2 \neq 1$, then by Theorem \ref{deg_prop_1.1}(b), we have $\deg(x) = |G|-|Z(H,G)|-|C_G(x)|-1 = 1$, which will lead to \eqref{induced-eq-1} (and eventually to a contradiction) or $\deg(x) = |G|-|Z(H,G)|-2|C_G(x)|-1 = 1$. That is, \begin{equation}\label{induced-eq-3}
\text{or }|G|-|Z(H,G)|-2|C_G(x)| = 2.
\end{equation}
Therefore, $|Z(H,G)| = 1$ or $2$. Thus if $|Z(H,G)| = 1$ then \eqref{induced-eq-3} gives $|G| = 9$, which is a contradiction since $G$ is non-abelian. Again if $|Z(H,G)| = 2$ then \eqref{induced-eq-3} gives  $|C_G(x)| = 2$ or $4$. Therefore, $|G| = 8$  or $|G| = 12$. If $|G| = 8$ then we get a contradiction as shown in Case 2 above. If $|G| = 12$ then  $G \cong D_{12}$ or $Q_{12}$,  since $|Z(H,G)| = 2$.
In both the cases we must have  $H = G$ and hence, by {\rm\cite[Theorem 2.5]{TEJ14}}, we get a contradiction.

 Now we assume that $x \in G \setminus H$ and consider  the following cases.
 
\noindent \textbf{Case 1:} If $g = 1$, then by Theorem \ref{deg_prop_2.1}(a), we have  $\deg(x) = |H|-|C_H(x)| = 1$. Therefore, $|H| = 2$, a contradiction.

\noindent \textbf{Case 2:} If $g \neq 1$ and $g^2 = 1$, then by Theorem \ref{deg_prop_2.1}(c), we have $\deg(x) = |H|-|Z(H,G)|-|C_H(x)| = 1$. That is,  
\begin{equation}\label{induced-eq-4}
|H|-|C_H(x)| = 2.
\end{equation}
Therefore, 
$|H| = 3$ or $4$, a contradiction.

\noindent \textbf{Case 3:} If $g \neq 1$ and $g^2 \neq 1$, then by Theorem \ref{deg_prop_2.1}(b), we have $\deg(x) = |H|-|Z(H,G)|-|C_H(x)| = 1$, which  leads to \eqref{induced-eq-4} or $\deg(x) = |H|-|Z(H,G)|-2|C_H(x)| = 1$. That is,  
\begin{equation}\label{induced-eq-5}
|H| -2|C_H(x)| = 2.
\end{equation}
Therefore, $|C_H(x)| = 1$ or $2$. Thus if $|C_H(x)| = 1$ then \eqref{induced-eq-5} gives $|H| = 4$, a contradiction.
If $|C_H(x)| = 2$ then \eqref{induced-eq-5} gives  $|H| = 6$, a contradiction.
\end{proof}
The following theorems also give that the condition on $|H|$ in Theorem \ref{delta-not-tree} can not be removed completely.
\begin{theorem}\label{not_tree}
If $G$ is a non-abelian group of order $\leq 12$ and $g = 1$ then $\Delta_{H, G}^g$ is a tree if and only if $G \cong D_6$ or $D_{10}$ and $|H| = 2$. 
\end{theorem}

\begin{proof}
If $H$ is the trivial subgroup of $G$ then $\Delta_{H, G}^g$ is an empty graph. If $H=G$ then, by {\rm\cite[Theorem 2.5]{TEJ14}}, we have $\Delta_{H, G}^g$ is not a tree. So we examine only the proper subgroups of $G$, where  $G \cong D_6, D_8, Q_8, D_{10}, D_{12}, Q_{12}$ or $A_4$. We consider the following cases.

\noindent \textbf{Case 1:} $G \cong D_6 = \langle a, b : a^3 = b^2 = 1 \text{ and } bab^{-1}=a^{-1}\rangle$. If $|H| = 2$ then $H = \langle x\rangle$, where $x = b, ab$ and $a^2b$. 
We have $[x, y] \ne 1$ for all $y \in G \setminus Z(H, G)$. Therefore,  $\Delta_{H, D_6}^g$ is a star graph and hence, a tree. If $|H| = 3$ then $H = \{1, a, a^2\}$. In this case, the vertices $a$, $ab$, $a^2$ and $b$ make a cycle since $[ab, a] = a^2 = [a^2, ab]$ and $[a, b] = a = [b, a^2]$. 


\noindent \textbf{Case 2:}  $G \cong D_8 = \langle a, b : a^4 = b^2 = 1 \text{ and } bab^{-1} = a^{-1}\rangle$. If $|H| = 2$ then $H = Z(D_8)$ or $\langle a^rb \rangle$, where  $r = 1, 2, 3, 4$. Clearly $\Delta_{H,D_8}^g$ is an empty graph if $H = Z(D_8)$. If $H = \langle a^rb\rangle$ then, in each case,  $a^2$ is an isolated vertex in $G \setminus H$ (since $[a^2, a^rb]=1$). Hence, $\Delta_{H,D_{8}}^g$ is disconnected. If $|H|=4$ then  $H = \{1, a, a^2, a^3\}$, $\{1, a^2, b, a^2b\}$ or $\{1, a^2, ab, a^3b\}$. If $H = \{1, a, a^2, a^3\}$  then the vertices $ab$, $a$, $b$ and $a^3$ make a cycle; if $H=\{1, a^2, b, a^2b\}$ then the vertices $ab$, $b$, $a^3$ and $a^2b$ make a cycle; and if  $H=\{1, a^2, ab, a^3b\}$ then the vertices $ab$, $a$, $a^3b$ and $b$ make a cycle (since $[a,b]=[a^3,b]=[a^3,ab]=[a^3,a^2b]=[ab,a]=[a^2b,ab]=[a^3b,a]=[b,ab]=[b,a^3b]=a^2 \ne 1$). 

\noindent \textbf{Case 3:}  $G \cong Q_8 = \langle a, b : a^4 = 1, b^2=a^2 \text{ and } bab^{-1} = a^{-1}\rangle$. If $|H|=2$ then $H = Z(Q_8)$ and so $\Delta_{H,Q_8}^g$ is an empty graph. If $|H| = 4$ then $H = \{1, a, a^2, a^3\}$,  $\{1, a^2, b, a^2b\}$ and $\{1, a^2, ab, a^3b\}$. Again, if  $H = \{1, a, a^2, a^3\}$ then the vertices $a$, $b$, $a^3$ and $ab$ make a cycle; if
 $H = \{1, a^2, b, a^2b\}$ then the vertices $b$, $a^3b$, $a^2b$ and $a^3$ make a cycle; and if   
  $H = \{1, a^2, ab, a^3b\}$ then the vertices $ab$, $a$, $a^3b$ and $a^2b$ make a cycle (since $[a, b] = [b, a^3] = [a^3, ab] = [ab, a] = [b, a^3b] = [a^3b, a^2b] = [a^2b, a^3] = [a, a^3b] = [a^2b, ab] = a^2 \ne 1$). 

\noindent \textbf{Case 4:}  $G \cong D_{10} = \langle a, b : a^5 = b^2 = 1 \text{ and } bab^{-1} = a^{-1}\rangle$. If $|H| = 2$ then $H = \langle a^rb\rangle$, for every integer $r$ such that $1 \leq r \leq 5$. For each case of $H$, $\Delta_{H, D_{10}}^g$ is a star graph since $[a^rb, x] \ne g$ for all $x \in G \setminus H$. If $|H| = 5$ then   $H = \{1, a, a^2, a^3, a^4\}$. In this case,  the vertices $a$, $ab$, $a^3$ and $a^3b$ make a cycle in $\Delta_{H, D_{10}}^g$ since $[a, ab] = a^3 \ne 1$, $[ab, a^3] = a \ne 1$, $[a^3, a^3b] = a^4 \ne 1$ and $[a^3b, a] = a^2  \ne 1$.

\noindent \textbf{Case 5:} $G \cong D_{12} = \langle a, b : a^6 = b^2 = 1 \text{ and } bab^{-1} = a^{-1}\rangle$. If $|H| = 2$ then $H = Z(D_{12})$ or $\langle a^rb\rangle$, for every integer $r$ such that $1 \leq r \leq 6$. 
If $|H| = 3$ then $H = \{1, a^2, a^4\}$. If $|H| = 4$ then $H = \{1, a^3, b, a^3b\}$,  $\{1, a^3, ab, a^4b\}$ or $\{1,a^3,a^2b,a^5b\}$. If $|H|=6$ then $H = \{1, a, a^2, a^3, a^4, a^5\}$, $\{1, a^2, a^4, b, a^2b, a^4b\}$ or $\{1,a^2,a^4,ab,a^3b,a^5b\}$. Note that $\Delta_{H,D_{12}}^g$ is an empty graph if $H = Z(D_{12})$. If $H = \langle a^rb \rangle$ (for $1 \leq r \leq 6$),  $\{1, a^2, a^4\}$,  $\{1, a^2, a^4, b, a^2b, a^4b\}$ or $\{1, a^2, a^4, ab, a^3b, a^5b\}$ then in each case the vertex $a^3$ is an isolated vertex in $G \setminus H$ (since $a^3 \in Z(D_{12})$) and hence $\Delta_{H,D_{12}}^g$ is disconnected. We have  $[a, b] = [b, a^5] = [a, ab] = [a^4, a^4b] = [a^5b, a^2] = [b,a^2] =[a^2b,a^5] = a^4 \ne 1$ and $[a^5, a^5b] = [a^5b, a] = [ab, a^4] = [a^4b, a] = [a^2, a^2b] = [a^2b, a] = a^2 \ne 1$. Therefore, if  $H = \{1, a^3, b, a^3b\}$ then the vertices $a$, $b$, $a^5$ and $a^5b$ make a cycle; if $H=\{1,a^3,ab,a^4b\}$ then the vertices $a$, $ab$, $a^4$ and $a^4b$ make a cycle; if $H=\{1, a^3, a^2b, a^5b\}$ then the vertices $a^2$, $a^2b$, $a^5$ and $a^5b$ make a cycle; and if  $H = \{1, a, a^2, a^3, a^4, a^5\}$ then the vertices $a$, $b$, $a^2$ and $a^2b$ make a cycle. 

\noindent \textbf{Case 6:}  $G \cong A_4 = \langle a, b : a^2 = b^3 = (ab)^3 = 1 \rangle$. If $|H| = 2$ then $H = \langle a\rangle$,  $\langle bab^2 \rangle$ or $\langle b^2ab\rangle$. Since the elements $a, bab^2$ and $b^2ab$  commute among themselves, in each case the remaining two elements in $G \setminus H$ remain isolated and hence $\Delta_{H,A_{4}}^g$ is disconnected. If $|H| = 3$ then $H = \langle x  \rangle$, where $x = b$, $ab$,  $ba$, $aba$. In each case, the vertices $x$, $a$, $x^{-1}$ and $bab^2$ make a cycle.
If $|H| = 4$ then $H = \{1, a, bab^2, b^2ab\}$. In this case, the vertices $a$, $b$, $bab^2$ and $ab$ make a cycle.

\noindent \textbf{Case 7:} $G \cong Q_{12} = \langle a, b : a^6  = 1, b^2 = a^3 \text{ and } bab^{-1} = a^{-1}\rangle$. If $|H| = 2$ then  $H = Z(Q_{12})$ and so $\Delta_{H,Q_{12}}^g$ is an empty graph. If $|H|=3$ then $H = \{1, a^2, a^4\}$. In this case,  $a^3$ is an isolated vertex in $G \setminus H$ (since $a^3 \in Z(D_{12})$) and so $\Delta_{H,Q_{12}}^g$ is disconnected. If $|H|= 4$ then $H = \{1, a^3, b, a^3b\}$, $\{1, a^3, ab, a^4b\}$ or $\{1,a^3,a^2b,a^5b\}$. If $|H| = 6$ then $H =\{1, a, a^2, a^3, a^4, a^5\}$. We have $[a, b] = [a, ab] = [a^4, a^4b] = [a^5b, a^2] = [b, a^2] = [b, a^5] =[a^2b,a^5] = a^4 \ne 1$ and $[a^5, a^5b] = [a^5b, a] = [ab, a^4] = [a^4b, a] = [a^2, a^2b] = [a^2b, a] = a^2 \ne 1$. Therefore, if $H = \{1, a^3, b, a^3b\}$ then the vertices $a$, $b$, $a^5$ and $a^5b$ make a cycle; if $H = \{1, a^3, ab, a^4b\}$ then the vertices $a$, $ab$, $a^4$ and $a^4b$ make a cycle; if  $H = \{1, a^3, a^2b, a^5b\}$ then the vertices $a^2$, $a^2b$, $a^5$ and $a^5b$ make a cycle; and if $H = \{1, a, a^2, a^3, a^4, a^5\}$ then the vertices $a$, $b$, $a^2$ and $a^2b$ make a cycle.  This completes the proof.
\end{proof}

\begin{theorem}\label{not_a_tree2}
If $G$ is a non-abelian group of order $\leq 12$ and $g \ne 1$ then $\Delta_{H, G}^g$ is a tree if and only if $g^2 = 1$, $G \cong A_4$ and $|H|=2$ such that $H = \langle g \rangle$.
\end{theorem}
\begin{proof}
If $H$ is the trivial subgroup of $G$ then $\Delta_{H, G}^g$ is an empty graph. 
If $H=G$ then, by {\rm\cite[Theorem 2.5]{TEJ14}}, we have $\Delta_{H, G}^g$ is not a tree. So we examine only the proper subgroups of $G$, where  $G \cong D_6, D_8, Q_8, D_{10}, D_{12}, Q_{12}$ or $A_4$. We consider the following two cases.

\noindent {\bf Case 1:}  $g^2 = 1$

In this case  $G \cong D_8$, $Q_8$ or $A_4$. 
If $G \cong D_8=\langle a, b: a^4 = b^2 = 1 \text{ and } bab^{-1} = a^{-1}\rangle$ then $g = a^2$ and $|H|= 2, 4$. If $|H|=2$ then $H = Z(D_8)$ or $\langle a^rb \rangle$, for every integer $r$ such that $1 \leq r \leq 4$. For $H = Z(D_8)$, $\Delta_{H, D_8}^g$ is an empty graph. For $H = \langle a^rb \rangle$, in each case  $a$ is an isolated vertex in $G \setminus H$ (since $[a, a^rb]=a^2$) and hence, $\Delta_{H, D_{8}}^g$ is disconnected. If $|H|=4$ then $H = \{1, a, a^2, a^3\}$, $\{1, a^2, b, a^2b\}$ or $\{1, a^2, ab, a^3b\}$. For $H = \{1, a, a^2, a^3\}$, $b$ is an isolated vertex in $G \setminus H$ (since $[a, b]=a^2=[a^3, b]$) and hence, $\Delta_{H, D_{8}}^g$ is disconnected. If $H = \{1, a^2, b, a^2b\}$ or $\{1, a^2, ab,  a^3b\}$ then $a$ is an isolated vertex in $G \setminus H$ (since $[a,a^rb] = a^2$ for every integer $r$ such that $1 \leq r \leq 4$) and hence, $\Delta_{H, D_{8}}^g$ is disconnected.

  If $G \cong Q_8 = \langle a, b : a^4 = 1, b^2 = a^2 \text{ and } bab^{-1} = a^{-1}\rangle$ then $g = a^2$ and $|H|= 2, 4$. If $|H| = 2$ then $H = Z(Q_8)$ and hence  $\Delta_{H, Q_8}^g$ is an empty graph. If $|H| = 4$ then $H = \{1, a, a^2, a^3\}$, $\{1, a^2, b, a^2b\}$ or $\{1, a^2, ab, a^3b\}$. In each case, vertices of $H \setminus Z(H, G)$ commute with each other and commutator of these vertices and those of $G \setminus H$ equals $a^2$. Hence, the vertices in $G \setminus H$ remain isolated and so  $\Delta_{H, Q_{8}}^g$ is disconnected.
	


If $G \cong A_4 = \langle a, b : a^2 = b^3 = (ab)^3 = 1\rangle$ then  $g \in \{a, bab^2, b^2ab\}$ and $|H|= 2, 3, 4$. 
If $|H|=2$ then $H = \langle a\rangle$, $\langle bab^2\rangle$ or $\langle b^2ab \rangle$.
If $H = \langle g \rangle$ then $\Delta_{H,A_{4}}^g$ is a star graph because $[g,x]\ne g$ for all $x \in G \setminus H$ and hence a tree; otherwise $\Delta_{H,A_{4}}^g$ is not a tree as shown in Figures 1.1--1.6. 
If $|H| = 3$ then $H = \langle x\rangle$, where $x = b,  ab, ba, aba$ or their inverses. We have $[x, x^{-1}] = 1$, $[x, g] \ne g$ and $[x^{-1}, g] \ne g$. Therefore,   $x$, $x^{-1}$ and $g$ make a triangle for each such subgroup in the graph $\Delta_{H, A_{4}}^g$.
If $|H|=4$ then $H = \{1, a, bab^2, b^2ab\}$. Since $H$ is abelian, the vertices $a$, $bab^2$ and $b^2ab$ make a triangle in the graph $\Delta_{H, A_{4}}^g$.

\noindent {\bf Case 2:}  $g^2 \ne 1$

In this case $G \cong D_6$, $D_{10}$, $D_{12}$ or $Q_{12}$.

If $G \cong D_6 = \langle a, b : a^3 = b^2 = 1 \text{ and } bab^{-1} = a^{-1}\rangle$ then  $g \in \{a, a^2\}$ and $|H| = 2, 3$. We have $\Delta_{H,D_6}^a = \Delta_{H,D_6}^{a^2}$ since  $a^{-1} = a^2$. If $|H| = 2$ then $H = \langle x\rangle$, where $x = b, ab$ and $a^2b$. 
We have $[x, y] \in \{g, g^{-1}\}$ for all $y \in G \setminus H$ and so $\Delta_{H, D_{6}}^g$ is an empty graph. If $|H| = 3$ then $H = \{1, a, a^2\}$. In this case, the vertices of $G \setminus H$ remain isolated since for $y \in G \setminus H$ we have $[a, y], [a^2, y] \in \{g, g^{-1}\}$.
If $G \cong D_{10}=\langle a, b : a^5 = b^2 = 1 \text{ and } bab^{-1} = a^{-1}\rangle$ then $g \in \{a, a^2, a^3, a^4\}$ and $|H| = 2, 5$.
We have $\Delta_{H,D_{10}}^a = \Delta_{H,D_{10}}^{a^4}$ and $\Delta_{H,D_{10}}^{a^2} = \Delta_{H,D_{10}}^{a^3}$  since  $a^{-1}=a^4$ and $(a^2)^{-1}=a^3$. 
Suppose that $|H| = 2$. Then $H = \langle a^rb \rangle$, for every integer $r$ such that $1 \leq r \leq 5$. If $g = a$ then for each subgroup $H$,   $a^2$ is an isolated vertex in $\Delta_{H,D_{10}}^g$ (since $[a^2,a^rb]=a^4$ for every integer $r$ such that $1 \leq r \leq 5$). If $g = a^2$ then for each subgroup $H$,  $a$ is an isolated vertex in $\Delta_{H,D_{10}}^g$ (since $[a,a^rb]=a^2$ for every integer $r$ such that $1 \leq r \leq 5$). Hence, $\Delta_{H,D_{10}}^g$ is disconnected for each $g$ and each subgroup $H$ of order $2$. Now suppose that $|H|=5$. Then we have $H=\{1,a,a^2,a^3,a^4\}$. In this case,  the vertices $a$, $a^2$, $a^3$ and $a^4$ make a cycle in $\Delta_{H,D_{10}}^g$ for each $g$ as they commute among themselves.

If $G \cong D_{12}=\langle a,b~|~a^6=b^2=1 \text{ and } bab^{-1}=a^{-1}\rangle$ then  $g \in \{a^2,a^4\}$ and $|H|=2, 3, 4, 6$. We have $\Delta_{H,D_{12}}^{a^2} = \Delta_{H,D_{12}}^{a^4}$  since $(a^2)^{-1}=a^4$. Suppose that $|H|=2$ then $H = Z(D_{12})$ or $\langle a^rb \rangle$, for every integer $r$ such that $1 \leq r \leq 6$. For $H = Z(D_{12})$, $\Delta_{H,D_{12}}^g$ is an empty graph. For $H = \langle a^rb \rangle$, in each case  $a$ is an isolated vertex in $G \setminus H$ (since $[a,a^rb]=a^2$ for every integer $r$ such that $1 \leq r \leq 6$) and hence, $\Delta_{H,D_{12}}^g$ is disconnected.  If $|H|=3$ then  $H=\{1,a^2,a^4\}$. In this case, the vertices $a$, $a^2$ and $a^4$ make a triangle in $\Delta_{H,D_{12}}^g$ since they commute among themselves. If $|H|=4$ then $H = \{1,a^3,b,a^3b\}$, $\{1,a^3,ab,a^4b\}$ or $\{1,a^3,a^2b,a^5b\}$. For all these  $H$,  $a$ is an isolated vertex in $G \setminus H$ (since $[a,a^rb]=a^2$ for every integer $r$ such that $1 \leq r \leq 6$) and hence, $\Delta_{H,D_{12}}^g$ is disconnected. If $|H|=6$ then $H=\{1,a,a^2,a^3,a^4,a^5\}$, $\{1,a^2,a^4,b,a^2b,a^4b\}$ or $\{1,a^2,a^4,ab,a^3b,a^5b\}$. For all these $H$ the vertices $a$, $a^2$, $a^4$ and $a^5$ make a cycle in $\Delta_{H,D_{12}}^g$ since they commute among themselves. 

If $G \cong Q_{12}=\langle a,b~|~a^6=1, b^2=a^3 \text{ and } bab^{-1}=a^{-1}\rangle$ then  $g \in \{a^2,a^4\}$ and $|H|=2, 3, 4, 6$. We have $\Delta_{H,D_{12}}^{a^2} = \Delta_{H,D_{12}}^{a^4}$  since $(a^2)^{-1}=a^4$. If $|H|=2$ then $H = Z(Q_{12})$ and so  $\Delta_{H,Q_{12}}^g$ is an empty graph. If $|H|=3$ then $H=\{1,a^2,a^4\}$. In this case, the vertices $a$, $a^2$ and $a^4$ make a triangle in $\Delta_{H,Q_{12}}^g$ since they commute among themselves. If $|H|=4$ then $H = \{1,a^3,b,a^3b\}$, $\{1,a^3,ab,a^4b\}$ or $\{1,a^3,a^2b,a^5b\}$. For all these $H$,  $a$ is an isolated vertex in $G \setminus H$ (since $[a,a^rb]=a^2$ for every integer $r$ such that $1 \leq r \leq 6$) and hence, $\Delta_{H,Q_{12}}^g$ is disconnected. If $|H|=6$ then $H=\{1,a,a^2,a^3,a^4,a^5\}$. In this case, the vertices $a$, $a^2$, $a^4$ and $a^5$ make a cycle in $\Delta_{H,Q_{12}}^g$ since they commute among themselves.
\end{proof}

\section{Connectivity of diameter}
Connectivity of $\Delta_{G}^g$ was studied in \cite{NEGJ16,NEGJ17,NEM18}. It was conjectured that  diameter of $\Delta_{G}^g$ is equal to $2$ if $\Delta_{G}^g$ is connected. In this section we discuss  the connectivity of $\Delta_{H, G}^g$. In general, $\Delta_{H, G}^g$ is not connected. For any two vertices $x$ and $y$, we write $x \sim y$ and $x \nsim y$ respectively to mean that they are adjacent or not.   We write $d(x, y)$ and $\diam(\Delta_{H, G}^g)$ to denote the distance between the vertices $x, y$ and  diameter of $\Delta_{H, G}^g$ respectively. 

\begin{theorem}\label{Ind-diam-1}
If $g$ is a non-central element of $G$ such that $g \in H$ and $g^2 = 1$ then $\diam(\Delta^g_{H,G}) = 2$. 
\end{theorem}
\begin{proof}
Let $x \ne g$ be any vertex of $\Delta^g_{H,G}$. Then $[x,g] \ne g$ which implies $[x,g] \ne g^{-1}$ since $g^2 = 1$. Since $g \in H$, if follows that $x \sim g$. Therefore, $d(x,g) = 2$ and hence $\diam(\Delta^g_{H,G}) = 2$.
\end{proof}
\begin{lemma}\label{Ind-diam-2}
Let $g \in H \setminus Z(H,G)$ such that $g^2 \ne 1$ and $o(g) \neq 3$. If $x \in G \setminus Z(H,G)$ and $x \nsim g$ then  $x \sim g^2$.
\end{lemma}
\begin{proof}
Since $g \ne 1$ and $x \nsim g$ it follows that   $[x, g] = g^{-1}$. We have 
\begin{equation}\label{Ind-eq-01}
[x,g^2]= [x, g][x, g]^g = g^{-2} \ne g, g^{-1}.
\end{equation}
If $g^2 \in Z(H,G)$ then, by \eqref{Ind-eq-01},  we have $g^{-2} = [x,g^2] = 1$; a contradiction. Therefore, $g^2 \in H\setminus  Z(H,G)$. Hence, $x \sim g^2$. 
\end{proof}
\begin{theorem}
Let $g \in H \setminus Z(H,G)$ and $o(g) \ne 3$. Then $\diam(\Delta^g_{H,G}) \leq 3$. 
\end{theorem} 
\begin{proof}
If $g^2 = 1$ then, by Theorem \ref{Ind-diam-1}, we have $\diam(\Delta^g_{H,G}) = 2$. Therefore, we assume that $g^2 \ne 1$. Let $x, y$ be any two vertices of $\Delta^g_{H,G}$ such that $x \nsim y$. Therefore, $[x, y] = g$ or $g^{-1}$. If $x \sim g$ and $y \sim g$ then $x \sim g \sim y$ and so $d(x, y) = 2$. If $x \nsim g$ and $y \nsim g$ then, by Lemma \ref{Ind-diam-2}, we have $x \sim g^2 \sim y$ and so $d(x, y) = 2$. Therefore, we shall  not consider these two situations in the following cases.

\noindent \textbf{Case 1:}  $x, y \in H$ 

 Suppose that one of $x, y$ is adjacent to $g$ and the other is not. Without any loss we assume that $x \nsim g$ and $y \sim g$. Then
$[x, g] = g^{-1}$ and $[y, g] \ne g, g^{-1}$.
By  Lemma \ref{Ind-diam-2}, we have $x \sim g^2$.

  Consider the element $yg \in H$. If $yg \in Z(H,G)$ then $[y, g^2] = 1 \ne g, g^{-1}$. Therefore, $x \sim g^2 \sim y$ and so $d(x, y) = 2$. 

   If $yg \notin Z(H,G)$ then  we have $[x,yg] = [x, g][x, y]^g = g^{-1}[x, y]^g \ne g, g^{-1}$.  
 Also, $[y, yg] = [y, g] \ne g, g^{-1}$. Hence, $x \sim yg \sim y$ and so $d(x, y) = 2$.

\noindent \textbf{Case 2:} One of $x, y$ belongs to $H$ and the other does not.

Without any loss assume that $x \in H$ and $y \notin H$. 
If $x \nsim g$ and $y \sim g$ 
then, by  Lemma \ref{Ind-diam-2}, we have $x \sim g^2$. Also, $[g, g^2] = 1 \ne g, g^{-1}$ and so $g^2 \sim g$. Therefore, $x \sim g^2 \sim g \sim y$ and hence   $d(x,y) \leq 3$.  
If $x \sim g$ and $y \nsim g$ then 
$[x, g] \ne g,g^{-1}$ and $[y, g] =  g^{-1}$. By  Lemma \ref{Ind-diam-2}, we have $y \sim g^2$. Consider the element $xg \in H$.  If $xg \in Z(H,G)$ then $[x, g^2] = 1 \ne g, g^{-1}$. Therefore, $x \sim g^2$ and  so $y \sim g^2 \sim x$. Thus $d(x, y) = 2$. 

 If $xg \notin Z(H,G)$ then  we have $[y, xg] = [y, g][y, x]^g = g^{-1}[y, x]^g \ne g, g^{-1}$. 
 Also, $[x, xg] = [x, g] \ne g, g^{-1}$. Hence, $y \sim xg \sim x$ and so $d(x, y) = 2$.

\noindent \textbf{Case 3:} $x, y \notin H$.

 Suppose that one of $x, y$ is adjacent to $g$ and the other is not. Without any loss we assume that $x \nsim g$ and $y \sim g$. Then,
by  Lemma \ref{Ind-diam-2}, we have $x \sim g^2$.
Also, $[g, g^2] = 1 \ne g, g^{-1}$ and so $g^2 \sim g$. Therefore, $x \sim g^2 \sim g \sim y$ and hence   $d(x,y) \leq 3$.

Thus $d(x,y) \leq 3$ for all $x, y \in G \setminus Z(H,G)$.  Hence the result follows. 
\end{proof}

The rest part of this paper is devoted to the study of connectivity  of $\Delta_{H, D_{2n}}^g$,
where $D_{2n} = \langle a, b: a^n = b^2 = 1, bab^{-1} = a^{-1} \rangle$ is the dihedral group of order $2n$. It is well-known that $Z(D_{2n}) = \{1\}$, the commutator subgroup $D_{2n}' = \langle a \rangle$ if $n$ is odd and $Z(D_{2n}) = \{1,a^\frac{n}{2}\}$ and $D_{2n}' = \langle a^2 \rangle$ if $n$ is even. By  \cite[Theorem 4]{NEGJ17}, it follows that $\Delta_{H, D_{2n}}^g$ is disconnected if $n = 3, 4, 6$. Therefore, we consider  $n \geq 8$ and $n \geq 5$ according as $n$ is even or odd in the following results.

\begin{theorem}
Consider the graph $\Delta_{H, D_{2n}}^g$, where $n \, (\geq 8)$ is even.   
\begin{enumerate}
\item If $H =  \langle a \rangle $ then $\Delta_{H, D_{2n}}^g$ is connected and $\diam(\Delta_{H, D_{2n}}^g) = 2$. 

\item Let $H =  \langle a^\frac{n}{2}, a^rb   \rangle $  for $0 \leq r < \frac{n}{2}$. Then $\Delta_{H, D_{2n}}^g$ is connected with diameter $2$ if $g = 1$ and  $\Delta_{H, D_{2n}}^g$ is not connected if $g \ne 1$. 

\item If $H = \langle a^rb \rangle$ for $1 \leq r \leq n$ then $\Delta_{H, D_{2n}}^g$ is not connected. 
\end{enumerate}
\end{theorem}
\begin{proof}
Since $n$ is even we have $g = a^{2i}$ for  $1 \leq i \leq \frac{n}{2}$.

\noindent (a) \textbf{Case 1:} $g=1$

Since $H$ is abelian, the induced subgraph of $\Delta_{H, D_{2n}}^g$ on $H \setminus Z(H, D_{2n})$ is empty. So we need to see the adjacency of these vertices with those in $D_{2n} \setminus H$. Suppose that $[a^rb, a^j]=1$ and $[b,a^j]=1$ for every integers $r,j$ such that $1 \leq r,j \leq n-1$.  Then $a^{2j} = a^0$ or $a^{n}$ and so $j=0$ or $j=\frac{n}{2}$.
Therefore, every vertex in $H \setminus Z(H, D_{2n})$ is adjacent to all the vertices in $D_{2n} \setminus H$.   Thus $\Delta_{H, D_{2n}}^g$ is connected and $\diam(\Delta_{H, D_{2n}}^g) = 2$.

\noindent \textbf{Case 2:} $g \ne 1$

Since $H$ is abelian, the induced subgraph of $\Delta_{H, D_{2n}}^g$ on $H \setminus Z(H, D_{2n})$ is a complete graph. Therefore, it is sufficient to prove that no vertex in $D_{2n} \setminus H$ is isolated.  
If $g \neq g^{-1}$ then $g \neq a^{\frac{n}{2}}$.
Suppose that $[a^rb, a^j]=g$ and $[b,a^j]=g$ for every integers $r,j$ such that $1 \leq r,j \leq n-1$.  Then $a^{2j} = a^{2i}$ and so $j=i$ or $j=\frac{n}{2}+i$. If $[a^rb, a^j]=g^{-1}$ and $[b,a^j]=g^{-1}$ for every integers $r,j$ such that $1 \leq r,j \leq n-1$ then $a^{2j} = a^{n -2i}$ and so $j=n-i$ or $j=\frac{n}{2}-i$. Therefore, there exists an  integer $j$ such that $1 \leq j \leq n-1$ and $j \neq i, \frac{n}{2}+i, n-i \text{ and } \frac{n}{2}-i$ for which $a^j$ is adjacent to all the  vertices in $D_{2n} \setminus H$.
If $g = g^{-1}$ then $g = a^{\frac{n}{2}}$. Suppose that $[a^rb, a^j]=g$ and $[b,a^j]=g$ for every integers $r,j$ such that $1 \leq r,j \leq n-1$ then $a^{2j} = a^{\frac{n}{2}}$ and so $j= \frac{n}{4}$ or $j=\frac{3n}{4}$.  Therefore, there exists an integer $j$ such that $1 \leq j \leq n-1$ and $j \neq \frac{n}{4} \text{ and } \frac{3n}{4}$ for which $a^j$ is adjacent to all the  vertices in $D_{2n} \setminus H$. Thus $\Delta_{H, D_{2n}}^g$ is connected and $\diam(\Delta_{H, D_{2n}}^g) = 2$.

\noindent (b)  \textbf{Case 1:} $g=1$

We have $[a^{\frac{n}{2}+r}b,a^rb] = 1$ for every integer $r$ such that $1 \leq r \leq n-1$. Therefore, the induced subgraph of $\Delta_{H, D_{2n}}^g$ on $H \setminus Z(H, D_{2n})$ is empty. So we need to see the adjacency of these vertices with those in $D_{2n} \setminus H$. Suppose $[a^rb,a^i] = 1$ and $[a^{\frac{n}{2}+r}b,a^i] = 1$ for every integer $i$ such that $1 \leq i \leq n - 1$. Then $a^{2i} = a^{n}$ and so $i=\frac{n}{2}$.  
Therefore, for every integer $i$ such that $1 \leq i \leq n-1$ and $i \neq \frac{n}{2}$, $a^i$ is adjacent to both $a^rb$ and $a^{\frac{n}{2}+r}b$. Also we have $[a^sb,a^rb] = a^{2(s-r)}$ and $[a^{\frac{n}{2}+r}b,a^sb] = a^{2(\frac{n}{2}+r-s)}$ for every integer $s$ such that $1 \leq s \leq n - 1$. Suppose $[a^sb,a^rb] = 1$ and $[a^{\frac{n}{2}+r}b,a^sb] = 1$. Then $s=r$ or $s= \frac{n}{2}+r$. Therefore, for every integer $s$ such that $1 \leq s \leq n-1$ and $s \neq r, \frac{n}{2}+r$, $a^sb$ is adjacent to both $a^rb$ and $a^{\frac{n}{2}+r}b$. Thus $\Delta_{H, D_{2n}}^g$ is connected and $\diam(\Delta_{H, D_{2n}}^g) = 2$.

\noindent \textbf{Case 2:} $g \ne 1$

 If $H =  \langle a^\frac{n}{2}, a^rb   \rangle  = \{1, a^\frac{n}{2}, a^rb,a^{\frac{n}{2}+r}b\}$ for $0 \leq r < \frac{n}{2}$ then 
$H \setminus Z(H, D_{2n}) = \{a^rb, a^{\frac{n}{2}+r}b\}$.
We have $[a^rb,a^i] = a^{2i} = [a^{\frac{n}{2}+r}b,a^i]$ for every integer $i$ such that $1 \leq i \leq \frac{n}{2} - 1$. That is, $[a^rb,a^i] = g$ and  $[a^{\frac{n}{2}+r}b,a^i] = g$ for every integer $i$ such that $1 \leq i \leq \frac{n}{2} - 1$. Thus $a^i$ is an isolated vertex in $D_{2n} \setminus H$.
Hence, $\Delta_{H, D_{2n}}^g$ is not connected.

\noindent (c)  \textbf{Case 1:} $g = 1$

We have $[a^{\frac{n}{2}+r}b,a^rb] = 1$ for every integer $r$ such that $1 \leq r \leq n - 1$. Thus $a^{\frac{n}{2}+r}b$ is an isolated vertex in $D_{2n} \setminus H$. Hence, $\Delta_{H, D_{2n}}^g$ is not connected.

\noindent \textbf{Case 2:} $g \ne 1$ 

If $H =  \langle a^rb \rangle  = \{1,a^rb\}$ for $1 \leq r \leq n$ then 
$H \setminus Z(H, D_{2n}) = \{a^rb\}$.
We have $[a^rb,a^i] = a^{2i} = g$ for every integer $i$ such that $1 \leq i \leq \frac{n}{2} - 1$. Thus $a^i$ is an isolated vertex in $D_{2n} \setminus H$. Hence, $\Delta_{H, D_{2n}}^g$ is not connected.
\end{proof}

\begin{theorem}
Consider the graph $\Delta_{H, D_{2n}}^g$, where $n \, (\geq 8)$ and $\frac{n}{2}$  are even.
\begin{enumerate}
\item If $H = \langle a^2 \rangle$ then $\Delta_{H, D_{2n}}^g$ is connected with diameter $2$ if and only if $g \notin  \langle a^4 \rangle$. 

\item If $H =  \langle a^2, b \rangle$ or $ \langle a^2, ab \rangle$ then $\Delta_{H, D_{2n}}^g$ is connected with diameter $2$ if $g = 1$ and $\diam(\Delta_{H, D_{2n}}^g) \leq 3$ if $g \ne 1$. 
\end{enumerate}
\end{theorem}
\begin{proof}
Since $n$ is even,  we have $g = a^{2i}$ for  $1 \leq i \leq \frac{n}{2}$.

\noindent (a)  \textbf{Case 1:} $g = 1$

We know that the vertices in $H$ commutes with all the odd powers of $a$. That is, any vertex in  $\Delta_{H, D_{2n}}^g$ of the form $a^i$, where $i$ is an odd integer and $1 \leq i \leq n-1$,  is not adjacent with any vertex.
Hence, $\Delta_{H, D_{2n}}^g$ is not connected.

\noindent \textbf{Case 2:} $g \ne 1$

Since $H$ is abelian, the induced subgraph of $\Delta_{H, D_{2n}}^g$ on $H \setminus Z(H, D_{2n})$ is a complete graph. Also, the vertices in $H$ commutes with all the odd powers of $a$. That is, a vertex of the form $a^i$, where $i$ is an odd integer, in  $\Delta_{H, D_{2n}}^g$ is adjacent with all the vertices in $H$. 
We have  $[a^rb, a^{2i}] = a^{4i}$ and $[b,a^{2i}] = a^{4i}$ for every integers $r,i$ such that $1 \leq r \leq n-1$ and $1 \leq i \leq \frac{n}{2} - 1$. Thus, for $g \notin  \langle a^4 \rangle$, every vertex of $H$ is adjacent to the vertices of the form $a^rb$, where $1 \leq r \leq n$. Therefore, $\Delta_{H, D_{2n}}^g$ is a complete graph. Hence, it is connected  and $\diam(\Delta_{H, D_{2n}}^g) = 2$. 
Also, if $g = a^{4i}$ for some integer $i$ where $1 \leq i \leq \frac{n}{4} - 1$ (i,e., $g \in  \langle a^4 \rangle$)  then the vertices in $D_{2n} \setminus H$ will remain isolated. Hence $\Delta_{H, D_{2n}}^g$ is disconnected in this case. This completes the proof of part (a).

\noindent (b) \textbf{Case 1:} $g = 1$ 

Suppose that $H =  \langle a^2, b \rangle$.
Then $a^{2i} \nsim a^j$ but $a^{2i} \sim a^rb$ for all $i,j,r$ such that $1\leq i\leq \frac{n}{2}-1$, $i \neq \frac{n}{4}$;  $1\leq j\leq n-1$ is an odd number and $1\leq r\leq n$ because $[a^{2i},a^j]=1$ and $[a^{2i},a^rb]=a^{4i}$. We shall find a path to $a^j$, where $1\leq j\leq n-1$  is an odd number. We have $[a^j,b]= a^{2j} \neq 1$ and $a^j \in G \setminus H$ for all $j$ such that $1\leq j\leq n-1$ is an odd number. Therefore, $a^{2i} \sim b \sim a^j$. Hence, $\Delta_{H, D_{2n}}^g$ is connected and $\diam(\Delta_{H, D_{2n}}^g) = 2$.

If $H =  \langle a^2, ab   \rangle$ 
then $a^{2i} \nsim a^j$ but $a^{2i} \sim a^rb$ for all $i,j,r$ such that $1\leq i\leq \frac{n}{2}-1$, $i \neq \frac{n}{4}$;  $1\leq j\leq n-1$ is an odd number and $1\leq r\leq n$ because $[a^{2i},a^j]=1$ and $[a^{2i},a^rb]=a^{4i}$. We shall find a path to $a^j$, where $1\leq j\leq n-1$ is an odd number. We have $[a^j,ab]= a^{2j} \neq 1$ and $a^j \in G \setminus H$ for all $j$ such that $1\leq j\leq n-1$ is an odd number. Therefore, $a^{2i} \sim ab \sim a^j$. Hence, $\Delta_{H, D_{2n}}^g$ is connected and $\diam(\Delta_{H, D_{2n}}^g) = 2$.

\noindent \textbf{Case 2:} $g \ne 1$

We have $\langle a^2\rangle \subset H$.  Therefore, if $g \notin  \langle a^4\rangle$ then every vertex in $\langle a^2\rangle$ is adjacent to all other vertices in both the cases (as discussed in part (a)). Hence, $\Delta_{H, D_{2n}}^g$ is connected and $\diam(\Delta_{H, D_{2n}}^g) = 2$.  Suppose that $g = a^{4i}$ for some integer $i$, where $1 \leq i \leq \frac{n}{4} - 1$. 

Suppose that $H =  \langle a^2, b \rangle $. Then $a^{2i} \sim a^j$ but $a^{2i} \nsim a^rb$ for all $j,r$ such that $1\leq j\leq n-1$ is an odd number and $1\leq r\leq n$ because $[a^{2i},a^j]=1$ and $[a^{2i},a^rb]=a^{4i}$. We shall find a path between $a^{2i}$ and $a^rb$ for all $i,r$ such that $1\leq i\leq \frac{n}{2}-1$ and $1\leq r\leq n$. We have $[a^j,b]= a^{2j} \neq a^{4i}$ and $a^j \in G \setminus H$ for all $j$ such that $1\leq j\leq n-1$  is an odd number. Therefore, $a^{2i} \sim a^j \sim b$. Consider  the vertices of the form $a^rb$ where $1\leq r\leq n-1$. We have $[a^rb,b]=a^{2r}$. Suppose $[a^rb,b]=g$ then it gives $a^{2r}=a^{4i}$ which implies $r = 2i$ or $r=\frac{n}{2}+2i$. Therefore, $b \sim a^rb$ if and only if $r \neq 2i$ and $r \neq \frac{n}{2} + 2i$. So we have $a^{2i} \sim a^j \sim b \sim a^rb$, where $1\leq r\leq n-1$ and $r \neq 2i$ and $r \neq \frac{n}{2} + 2i$. Again we know that $a^{\frac{n}{2}+2i}b,a^{2i}b \in H$ and $[a^{\frac{n}{2} + 2i}b,a^{2i}b] = 1$, so $a^{\frac{n}{2}+2i}b \sim a^{2i}b$. If we are able to find a path between $a^j$ and any one of $a^{\frac{n}{2}+2i}b$ and $a^{2i}b$ then we are done. Now $[a^{2i}b,a^j] \neq a^{4i}$ and $[a^{\frac{n}{2}+2i}b,a^j] \neq a^{4i}$ for any odd number $j$ such that $1 \leq j \leq n-1$ so we have $a^{\frac{n}{2}+2i}b \sim a^j \sim a^{2i}b$. Thus 
$a^{2i} \sim a^j \sim a^{2i}b$,  $a^{2i} \sim a^j \sim a^{\frac{n}{2}+2i}b$, $a^rb \sim b \sim a^j \sim a^{2i}b$ and $a^rb \sim b \sim a^j \sim a^{\frac{n}{2}+2i}b$, where $1\leq r\leq n-1$ and $r \neq 2i$ and $r \neq \frac{n}{2} + 2i$. Hence, $\Delta_{H, D_{2n}}^g$ is connected and $\diam(\Delta_{H, D_{2n}}^g) \leq 3$.

If $H =  \langle a^2, ab   \rangle$ then $a^{2i} \sim a^j$ but $a^{2i} \nsim a^rb$ for all $j,r$ such that $1\leq j\leq n-1$  is an odd number and $1\leq r\leq n$ because $[a^{2i},a^j]=1$ and $[a^{2i},a^rb]=a^{4i}$. We shall find a path between $a^{2i}$ and $a^rb$ for all $i,r$ such that $1\leq i\leq \frac{n}{2}-1$ and $1\leq r\leq n$. We have $[a^j,ab]= a^{2j} \neq a^{4i}$ and $a^j \in G \setminus H$ for all $j$ such that $1\leq j\leq n-1$  is an odd number. So we have $a^{2i} \sim a^j \sim ab$. Consider the vertices of the form $a^rb$, where $2\leq r\leq n$. We have $[a^rb,ab]=a^{2(r-1)}$. Suppose $[a^rb,ab]=g$ then it gives $a^{2(r-1)}=a^{4i}$ which implies $r = 2i+1$ or $r=\frac{n}{2}+2i+1$. Therefore, $ab \sim a^rb$ if and only if $r \neq 2i+1$ and $r \neq \frac{n}{2} + 2i+1$. So we have $a^{2i} \sim a^j \sim ab \sim a^rb$, where $2\leq r\leq n$ and $r \neq 2i+1$ and $r \neq \frac{n}{2} + 2i+1$. Again we know that $a^{\frac{n}{2}+2i+1}b,a^{2i+1}b \in H$ and $[a^{\frac{n}{2} + 2i+1}b,a^{2i+1}b] = 1$, so $a^{\frac{n}{2}+2i+1}b \sim a^{2i+1}b$. If we are able to find a path between $a^j$ and any one of $a^{\frac{n}{2}+2i+1}b$ and $a^{2i+1}b$ then we are done. Now $[a^{2i+1}b,a^j] \neq a^{4i}$ and $[a^{\frac{n}{2}+2i+1}b,a^j] \neq a^{4i}$ for any odd number $j$ such that $1 \leq j \leq n-1$ so we have $a^{\frac{n}{2}+2i+1}b \sim a^j \sim a^{2i+1}b$. Thus 
$a^{2i} \sim a^j \sim a^{2i+1}b$, $a^{2i} \sim a^j \sim a^{\frac{n}{2}+2i+1}b$, $a^rb \sim ab \sim a^j \sim a^{2i+1}b$ and $a^rb \sim ab \sim a^j \sim a^{\frac{n}{2}+2i+1}b$, where $2\leq r\leq n$ and $r \neq 2i+1$ and $r \neq \frac{n}{2} + 2i+1$. Hence, $\Delta_{H, D_{2n}}^g$ is connected and $\diam(\Delta_{H, D_{2n}}^g) \leq 3$.
\end{proof}

\begin{theorem}
Consider the graph $\Delta_{H, D_{2n}}^g$, where $n \, (\geq 8)$ is even and $\frac{n}{2}$  is odd.   
\begin{enumerate}
\item If $H = \langle a^2 \rangle$ then $\Delta_{H, D_{2n}}^g$ is not connected if $g = 1$ and $\Delta_{H, D_{2n}}^g$ is connected with $\diam(\Delta_{H, D_{2n}}^g) = 2$ if $g \ne 1$. 

\item If $H =  \langle a^2, b \rangle$ or $ \langle a^2, ab \rangle$ then $\Delta_{H, D_{2n}}^g$ is not connected if $g = 1$ and $\Delta_{H, D_{2n}}^g$ is connected with $\diam(\Delta_{H, D_{2n}}^g) = 1$ if $g \ne 1$. 
\end{enumerate}
\end{theorem}

\begin{proof}
Since $n$ is even,  we have $g = a^{2i}$ for  $1 \leq i \leq \frac{n}{2}$.

\noindent  (a) \textbf{Case 1:} $g = 1$

We know that the vertices in $H$ commute with all the odd powers of $a$. That is, any vertex  of the form $a^i \in D_{2n} \setminus H$, where $i$ is an odd integer,  is not adjacent with any vertex in   $\Delta_{H, D_{2n}}^g$.  
Hence, $\Delta_{H, D_{2n}}^g$ is not connected.

\noindent \textbf{Case 2:} $g \ne 1$

Since $H$ is abelian, the induced subgraph of $\Delta_{H, D_{2n}}^g$ on $H \setminus Z(H, D_{2n})$ is a complete graph. Also, the vertices in $H$ commutes with all the odd powers of $a$. That is, a vertex of the form $a^i$, where $i$ is an odd integer, in  $\Delta_{H, D_{2n}}^g$ is adjacent with all the vertices in $H$. 
We claim that atleast one element of $H \setminus Z(H, D_{2n})$ is adjacent to all $a^rb$'s such that $1 \leq r \leq n$. 
Consider the  following cases.

\noindent \textbf{Subcase 1:} $g^3 \neq 1$

If $[g, a^rb] = g$, i.e.,  $[a^{2i}, a^rb] = a^{2i}$ for all $1 \leq i \leq \frac{n}{2} - 1$ and $1 \leq r \leq n$ then we get $g = a^{2i} = 1$, a contradiction. If $[g, a^rb] = g^{-1}$, i.e., $[a^{2i}, a^rb] = a^{n-2i}$ for all $1 \leq i \leq \frac{n}{2} - 1$ and $1 \leq r \leq n$ then  we get $g^3 = (a^{2i})^3 = a^{6i} = 1$,  
a contradiction. Therefore, $g$ is adjacent to all other vertices of the form $a^rb$ such that $1 \leq r \leq n$.

\noindent \textbf{Subcase 2:} $g^3 = 1$

If $[g, a^rb] =  g^{-1}$, i.e., $[a^{2i}, a^rb] =  a^{2i}$ then $[ga^{2}, a^rb] = g^{-1}a^{4}$ for all $1 \leq i \leq \frac{n}{2} - 1$ and $1 \leq r \leq n$. Now, if  $g^{-1}a^{4} = g^{-1}$ then $a^4 = 1$,  a contradiction since $a^n =1$ for $n \geq 8$. If $g^{-1}a^{4}= g$ then  
$a^{n-2i-4} = 1 $ for all $1 \leq i \leq \frac{n}{2} - 1$, 
which is a contradiction since $1 \leq i \leq \frac{n}{2} - 1$. Therefore, $ga^2$ is adjacent to all other vertices of the form $a^rb$ such that $1 \leq r \leq n$.

Thus there exists a vertex in $H \setminus Z(H, D_{2n})$ which is adjacent to all other vertices in $D_{2n}$.  Hence, $\Delta_{H, D_{2n}}^g$ is connected and $\diam(\Delta_{H, D_{2n}}^g) = 2$.

\noindent  (b) \textbf{Case 1:} $g = 1$ 

We know that the vertices in $H$ commute with the vertex $a^\frac{n}{2}$. That is, the vertex $a^\frac{n}{2} \in D_{2n} \setminus H$  is not adjacent with any vertex in   $\Delta_{H, D_{2n}}^g$. 
Hence, $\Delta_{H, D_{2n}}^g$ is not connected.

\noindent \textbf{Case 2:} $g \ne 1$

As shown in Case 2 of part (a), it can be seen that either $g$ or $ga^2$ is  adjacent to all other vertices. Hence, $\Delta_{H, D_{2n}}^g$ is connected and $\diam(\Delta_{H, D_{2n}}^g) = 1$.
\end{proof}

\begin{theorem}
Consider the graph $\Delta_{H, D_{2n}}^g$, where  $n (\geq 5)$ is odd.  
\begin{enumerate}
\item If $H =  \langle a \rangle$ then $\Delta_{H, D_{2n}}^g$ is connected and $\diam(\Delta_{H, D_{2n}}^g) = 2$.
\item If $H =  \langle a^rb   \rangle$, where $1 \leq r \leq n$,  then $\Delta_{H, D_{2n}}^g$ is connected with $\diam(\Delta_{H, D_{2n}}^g)$ $= 2$ if $g = 1$ and $\Delta_{H, D_{2n}}^g$ is not connected if $g \ne 1$.
\end{enumerate}
\end{theorem}

\begin{proof}
Since $n$ is odd,  we have $g = a^{i}$ for  $1 \leq i \leq n - 1$.

\noindent (a) \textbf{Case 1:} $g = 1$

Since $H$ is abelian, the induced subgraph of $\Delta_{H, D_{2n}}^g$ on $H \setminus Z(H, D_{2n})$ is empty. Therefore, we need to see the adjacency of these vertices with those in $D_{2n} \setminus H$. Suppose that $[a^rb, a^j]=1$ and $[b,a^j]=1$ for every integers $r,j$ such that $1 \leq r,j \leq n-1$.  Then $a^{2j} = a^{n}$ and so $j=\frac{n}{2}$, a contradiction. 
Therefore, for every integer $j$ such that $1 \leq j \leq n-1$, $a^j$ is adjacent to all the  vertices in $D_{2n} \setminus H$. Thus $\Delta_{H, D_{2n}}^g$ is connected and $\diam(\Delta_{H, D_{2n}}^g) = 2$.

\noindent \textbf{Case 2:} $g \ne 1$

 Since $H$ is abelian, the induced subgraph of $\Delta_{H, D_{2n}}^g$ on $H \setminus Z(H, D_{2n})$ is a complete graph. Therefore, it is sufficient to prove that no vertex in $D_{2n} \setminus H$ is isolated.  
Since $n$ is odd we have $g \ne g^{-1}$. 
If $[a^rb, a^j]=g$ and $[b,a^j]=g$ for every integers $r,j$ such that $1 \leq r,j \leq n-1$ then $j= \frac{i}{2}$ or $j= \frac{n+i}{2}$. If $[a^rb, a^j]=g^{-1}$ and $[b,a^j]=g^{-1}$ for every integers $r,j$ such that $1 \leq r,j \leq n-1$ then $j= \frac{n-i}{2}$ or $j= n-\frac{i}{2}$. Therefore, there exists an integer $j$ such that $1 \leq j \leq n-1$ and $j \neq \frac{i}{2}, \frac{n+i}{2}, \frac{n-i}{2}\text{ and } n-\frac{i}{2}$ for which $a^j$ is adjacent to all other vertices in $D_{2n} \setminus H$.
Thus $\Delta_{H, D_{2n}}^g$ is connected and $\diam(\Delta_{H, D_{2n}}^g) = 2$.

\noindent (b) \textbf{Case 1:} $g = 1$

We have $[a^rb, a^j] \neq 1$ and $[b,a^j] \neq 1$ for every integers $r,j$ such that $1 \leq r,j \leq n-1$. So $a^rb$ is adjacent to $a^j$ for every integer $j$ such that $1 \leq j \leq n-1$. Also we have $[a^sb,a^rb] = a^{2(s-r)}$ for every integers $r,s$ such that $1 \leq r,s \leq n$. Suppose $[a^sb,a^rb] = 1$ then $s=r$ as $s= \frac{n}{2}+r$ is not possible. Therefore, for every integers $r,s$ such that $1 \leq r,s \leq n$ and $s \neq r$, $a^sb$ is adjacent to $a^rb$. Thus $\Delta_{H, D_{2n}}^g$ is connected and $\diam(\Delta_{H, D_{2n}}^g) = 2$.

\noindent \textbf{Case 2:} $g \ne 1$

If $i$ is even then $[a^\frac{i}{2}, a^rb]=a^i=g$ and so the vertex $a^\frac{i}{2}$ remains isolated. If $i$ is odd then $n-i$ is even and we have $[a^\frac{n-i}{2}, a^rb]=a^{n-i}=g^{-1}$. Therefore, the vertex $a^\frac{n-i}{2}$ remains isolated. Hence, $\Delta_{H, D_{2n}}^g$ is not connected. 
\end{proof}

\begin{theorem}
Consider the graph $\Delta_{H, D_{2n}}^g$, where $n (\geq 5)$ is odd. 
\begin{enumerate}
\item If $H =  \langle a^d   \rangle$, where $d|n$ and $o(a^d) = 3$, then $\Delta_{H, D_{2n}}^g$ is connected with diameter $2$ if and only if $g = 1$.

\item If $H =  \langle a^d, b\rangle$, $\langle a^d, ab\rangle$ or $\langle a^d, a^2b\rangle$, where $d|n$ and $o(a^d) = 3$, then $\Delta_{H, D_{2n}}^g$ is connected with diameter $2$ if $g \neq 1, a^d, a^{2d}$.

\item If $H = \langle a^d, b\rangle$, where $d|n$ and $o(a^d) = 3$, then $\Delta_{H,D_{2n}}^g$ is connected and $\diam(\Delta_{H,D_{2n}}^g) = \begin{cases}
2, & \text{ if } g =1 \\
3, & \text{ if } g = a^d \text{ or } a^{2d}.
\end{cases}$
\end{enumerate}
\end{theorem}
\begin{proof}
(a) Given $H =  \{1, a^d, a^{2d}\}$. We have  $[a^d, a^{2d}] = 1$, $[a^d, a^rb] = a^{2d}$ and  $[a^{2d}, a^rb] = a^{4d} = a^d$ for all $r$ such that $1 \leq r \leq n$. Therefore, $g = 1, a^d$ or $a^{2d}$. If $g = a^d$ or $a^{2d}$ then $a^d \nsim a^rb$ and $a^{2d} \nsim a^rb$ for all $r$ such that $1 \leq r \leq n$. Thus $\Delta_{H, D_{2n}}^g$ is disconnected. If $g = 1$ then $a^d \sim a^rb$, $a^{2d} \sim a^rb$ and $a^d \sim a^rb \sim a^{2d}$  for all $r$ such that $1 \leq r \leq n$. Note that $a^d \nsim a^{2d}$. Hence, $\Delta_{H, D_{2n}}^g$ is connected with diameter $2$.

(b) If $g \neq 1, a^d, a^{2d}$ then $a^d$ is adjacent to all other vertices, as discussed in part (a). Hence, $\Delta_{H, D_{2n}}^g$ is connected and $\diam(\Delta_{H, D_{2n}}^g) = 2$.

(c) \textbf{Case 1:} $g =1$

 Since $n$ is odd, we have $2i \ne n$  for all integers $i$ such that $1 \leq i \leq n-1$. Therefore, if $g = 1$ then $b$ is adjacent to all other vertices because $[a^i, b] = a^{2i}$ and $[a^rb, b] = a^{2r}$ for all integers $i, r$ such that $1 \leq i, r \leq n-1$. Hence, $\Delta_{H,D_{2n}}^g$ is connected and $\diam(\Delta_{H,D_{2n}}^g) =2$.
 
\noindent\textbf{Case 2:} $g = a^d$ or $a^{2d}$

Since $[a^d, a^{2d}]=1$ we have $a^d \sim a^{2d}$. Also, all the vertices of the form $a^i$ commute among themselves, where $1 \leq i \leq n-1$. Therefore, $a^d \sim a^i \sim a^{2d}$ for all $1 \leq i \leq n-1$ such that $i \ne d, 2d$. Again, $[a^i, a^rb] = a^{2i} = [a^i, b]$ for all $1 \leq i, r \leq n-1$. If $[a^i, a^rb] = a^d$ or $a^{2d}$ for all $1 \leq r \leq n$, then $i = 2d$ or $d$ respectively. Therefore, $a^d \sim a^i \sim b$, $ a^d \sim a^i \sim a^db$, $a^d \sim a^i \sim a^{2d}b$, $a^{2d} \sim a^i \sim b$, $ a^{2d} \sim a^i \sim a^db$ and $a^{2d} \sim a^i \sim a^{2d}b$ for all $1 \leq i \leq n-1$ such that $i \ne d, 2d$. If $[a^rb, b] = a^d$ or $a^{2d}$ for all $1 \leq r \leq n-1$, then $a^{2r} = a^d$ or $a^{2d}$; which gives $r = 2d$ or $d$ respectively. Therefore, $a^d \sim a^i \sim b \sim a^rb$, $ a^{2d} \sim a^i \sim b \sim a^rb$, $a^db \sim a^i \sim b \sim a^rb$ and $a^{2d}b \sim a^i \sim b \sim a^rb$ for all $1 \leq i, r \leq n-1$ such that $i, r \ne d, 2d$. Hence, $\Delta_{H,D_{2n}}^g$ is connected and $\diam(\Delta_{H,D_{2n}}^g) =3$.
\end{proof}

\section*{Acknowledgment}
The first author would like to thank  DST for the INSPIRE Fellowship.

\end{document}